\documentclass[11pt]{amsart}
\usepackage{amssymb}
\usepackage{amsxtra}
\usepackage{graphics}
\usepackage{latexsym}
\usepackage{xcolor}
\usepackage{amsmath}
\usepackage{amssymb,amsthm,amsfonts}
\usepackage{mathtools}
\usepackage{amscd}
\usepackage[arrow, matrix, curve]{xy}
\usepackage{syntonly}

\title[Semistable Higgs bundles  are strongly Higgs semistable]{Semistable Higgs bundles of small ranks are strongly Higgs semistable}
\author[G.-T. Lan]{Guitang Lan}
\email{lan@uni-mainz.de}
\address{Institut f\"{u}r  Mathematik, Universit\"{a}t
Mainz, Mainz, 55099, Germany}

\author[M. Sheng]{Mao Sheng}
\email{msheng@ustc.edu.cn}
\address{School of Mathematical Sciences,
University of Science and Technology of China, Hefei, 230026, China}
\author[Y.-H. Yang]{Yanhong Yang}
\email{yangy@uni-mainz.de}
\address{Institut f\"{u}r Mathematik, Universit\"{a}t Mainz, Mainz, 55099, Germany}
\author[K. Zuo]{Kang Zuo}
\email{zuok@uni-mainz.de}
\address{Institut f\"{u}r  Mathematik, Universit\"{a}t
Mainz, Mainz, 55099, Germany}
\thanks{This work is supported by the SFB/TR 45 ``Periods, Moduli Spaces and Arithmetic of Algebraic
Varieties" of the DFG, and partially supported by University of Science and Technology of China.}

\begin{document}

\vfuzz0.5pc
\hfuzz0.5pc 
\theoremstyle{plain}
\newtheorem{thm}{Theorem}[section]
\newtheorem{theorem}[thm]{Theorem}
\newtheorem{lemma}[thm]{Lemma}
\newtheorem{corollary}[thm]{Corollary}
\newtheorem{proposition}[thm]{Proposition}
\newtheorem{addendum}[thm]{Addendum}
\newtheorem{variant}[thm]{Variant}
\theoremstyle{definition}
\newtheorem{lemma and definition}[thm]{Lemma and Definition}
\newtheorem{construction}[thm]{Construction}
\newtheorem{notations}[thm]{Notations}
\newtheorem{question}[thm]{Question}
\newtheorem{problem}[thm]{Problem}
\newtheorem{remark}[thm]{Remark}
\newtheorem{remarks}[thm]{Remarks}
\newtheorem{definition}[thm]{Definition}
\newtheorem{claim}[thm]{Claim}
\newtheorem{assumption}[thm]{Assumption}
\newtheorem{assumptions}[thm]{Assumptions}
\newtheorem{properties}[thm]{Properties}
\newtheorem{example}[thm]{Example}
\newtheorem{conjecture}[thm]{Conjecture}
\newtheorem{proposition and definition}[thm]{Proposition and Definition}
\numberwithin{equation}{thm}
\newcommand{\Spec}{\mathrm{Spec}}
\newcommand{\pP}{{\mathfrak p}}
\newcommand{\sA}{{\mathcal A}}
\newcommand{\sB}{{\mathcal B}}
\newcommand{\sC}{{\mathcal C}}
\newcommand{\sD}{{\mathcal D}}
\newcommand{\sE}{{\mathcal E}}
\newcommand{\sF}{{\mathcal F}}
\newcommand{\sG}{{\mathcal G}}
\newcommand{\sH}{{\mathcal H}}
\newcommand{\sI}{{\mathcal I}}
\newcommand{\sJ}{{\mathcal J}}
\newcommand{\sK}{{\mathcal K}}
\newcommand{\sL}{{\mathcal L}}
\newcommand{\sM}{{\mathcal M}}
\newcommand{\sN}{{\mathcal N}}
\newcommand{\sO}{{\mathcal O}}
\newcommand{\sP}{{\mathcal P}}
\newcommand{\sQ}{{\mathcal Q}}
\newcommand{\sR}{{\mathcal R}}
\newcommand{\sS}{{\mathcal S}}
\newcommand{\sT}{{\mathcal T}}
\newcommand{\sU}{{\mathcal U}}
\newcommand{\sV}{{\mathcal V}}
\newcommand{\sW}{{\mathcal W}}
\newcommand{\sX}{{\mathcal X}}
\newcommand{\sY}{{\mathcal Y}}
\newcommand{\sZ}{{\mathcal Z}}
\newcommand{\A}{{\mathbb A}}
\newcommand{\B}{{\mathbb B}}
\newcommand{\C}{{\mathbb C}}
\newcommand{\D}{{\mathbb D}}
\newcommand{\E}{{\mathbb E}}
\newcommand{\F}{{\mathbb F}}
\newcommand{\G}{{\mathbb G}}
\renewcommand{\H}{{\mathbb H}}
\newcommand{\I}{{\mathbb I}}
\newcommand{\J}{{\mathbb J}}
\renewcommand{\L}{{\mathbb L}}
\newcommand{\M}{{\mathbb M}}
\newcommand{\N}{{\mathbb N}}
\renewcommand{\P}{{\mathbb P}}
\newcommand{\Q}{{\mathbb Q}}
\newcommand{\Qbar}{\overline{\Q}}
\newcommand{\R}{{\mathbb R}}
\newcommand{\SSS}{{\mathbb S}}
\newcommand{\T}{{\mathbb T}}
\newcommand{\U}{{\mathbb U}}
\newcommand{\V}{{\mathbb V}}
\newcommand{\W}{{\mathbb W}}
\newcommand{\Z}{{\mathbb Z}}
\newcommand{\g}{{\gamma}}
\newcommand{\bb}{{\beta}}
\newcommand{\as}{{\alpha}}
\newcommand{\id}{{\rm id}}
\newcommand{\rk}{{\rm rank}}
\newcommand{\END}{{\mathbb E}{\rm nd}}
\newcommand{\End}{{\rm End}}
\newcommand{\Hom}{{\rm Hom}}
\newcommand{\Hg}{{\rm Hg}}
\newcommand{\tr}{{\rm tr}}
\newcommand{\Sl}{{\rm Sl}}
\newcommand{\Gl}{{\rm Gl}}
\newcommand{\Cor}{{\rm Cor}}

\newcommand{\SO}{{\rm SO}}
\newcommand{\OO}{{\rm O}}
\newcommand{\SP}{{\rm SP}}
\newcommand{\Sp}{{\rm Sp}}
\newcommand{\UU}{{\rm U}}
\newcommand{\SU}{{\rm SU}}
\newcommand{\SL}{{\rm SL}}

\newcommand{\ra}{\rightarrow}
\newcommand{\xra}{\xrightarrow}
\newcommand{\la}{\leftarrow}
\newcommand{\Nm}{\mathrm{Nm}}
\newcommand{\Gal}{\mathrm{Gal}}
\newcommand{\Res}{\mathrm{Res}}
\newcommand{\GL}{\mathrm{GL}}

\newcommand{\GSp}{\mathrm{GSp}}
\newcommand{\Tr}{\mathrm{Tr}}

\newcommand{\bA}{\mathbf{A}}
\newcommand{\bK}{\mathbf{K}}
\newcommand{\bM}{\mathbf{M}} 
\newcommand{\bP}{\mathbf{P}}
\newcommand{\bC}{\mathbf{C}}


\begin{abstract} 
We give a proof of the conjecture that  a  semistable Higgs bundle is strongly Higgs semistable in the case of small ranks, based upon the fact that there exists a gr-semistable Griffiths-transverse filtration on a $\nabla$-invariant semistable vector bundle. The latter is a generalization of   \cite[Theorem 2.5]{Simpson}.
\end{abstract}
\maketitle
\section{Introduction}
In  our previous work \cite{LSZ}, the notion of a strongly Higgs semistable  bundle has been introduced and it has been conjectured that a  semistable Higgs bundle is strongly Higgs semistable. In this note, we give a proof of the conjecture when the rank of Higgs bundle is smaller than the characteristic $p$ of the base field. 

The rank-$2$ case has been verified in \cite{LSZ}; the rank-$3$-and -$4$ case has been proved by Lingguang Li \cite{Li}. 
The key point of our proof  is to show the existence of a Griffiths-transverse filtration of a flat bundle such that the associated-graded Higgs bundle is semistable. 
In the case of a complex curve, the existence has been proved by Simpson \cite[Theorem 2.5]{Simpson}. We apply similar techniques and establish the existence of such a filtration in Theorem \ref{goodfil}.

Early this May, Adrian Langer kindly informed us that he could prove that every semistable Higgs bundle is strongly Higgs semistable. Since his manuscript, to which we have been looking forward, has not been available in any form, we decide to write down our own proof.

\section{A gr-semistable filtration in arbitrary characteristic}
In this section, we generalize in Theorem \ref{goodfil} a result of Simpson \cite[Theorem 2.5]{Simpson}, which states that over a smooth complex projective curve, every vector bundle with an integrable holomorphic connection admits a Griffiths-transverse filtration such that the associated-graded Higgs bundle is semistable. The generalization is proved by applying the same technique  as in \cite[Theorem 2.5]{Simpson}.

In this note,   vector (Higgs) bundles  always mean  torsion-free (Higgs) sheaves;    filtrations  only mean those consisting of saturated torsion-free subsheaves; 
 $X$ is a smooth projective variety over an algebraically closed field $k$;  the  semistability of vector (Higgs) bundles  is referred to the $\mu$-semistability with respect to an ample invertible sheaf $\sH$ over $X$. 

Let $(V, \nabla)$ be a flat vector bundle over $X$, i.e. the connection $\nabla: V\to V\otimes_{\sO_X}\Omega^1_X$ is integrable.
We first recall some notations from \cite{Simpson}.  A Griffiths-transverse filtration  is a decreasing filtration of $V$ by strict subbundles
\begin{align*}
V=F^0\supset F^1\supset F^2\cdots\supset F^n\supset F^{n+1}=0
\end{align*}
which satisfies the Griffiths transversality condition
\begin{align*}
\nabla: F^i\to F^{i-1}\otimes_{\sO_X}\Omega^1_X, \forall i.
\end{align*}
Put $E:=\text{Gr}_F(V):=\oplus_{i=0}^n E^i$ with $E^i=F^i/F^{i+1}$. By using $\nabla$, one can define an $\sO_X$-linear map
\begin{align*}
\theta: E^i\to E^{i-1}\otimes_{\sO_X}\Omega^1_X, \forall i.
\end{align*}
 $(E, \theta)$ is called to be the associated-graded Higgs bundle corresponding to $(V, \nabla, F^*)$. We say that  $(V, \nabla, F^*)$ is gr-semistable if $(E, \theta)$  is semistable as a Higgs bundle. Moreover, we need the notion of $\nabla$-semistability.

\begin{definition}
A flat bundle $(V, \nabla)$ is called to be $\nabla$-semistable if for every subbundle $V_1\subset V$ with $\nabla(V_1)\subset V_1\otimes_{\sO_X}\Omega^1_X$, we have $\mu(V_1)\leq \mu(V)$.
\end{definition}
The goal of this section is to prove the following theorem.
\begin{theorem}\label{goodfil}
Let $(V, \nabla)$  be a  $\nabla$-semistable flat bundle  over a smooth projective variety $X$ defined over an algebraically closed field $k$. Then there exists a Griffiths-transverse filtration $F^*$  such that  the associated-graded Higgs bundle corresponding to $(V, \nabla, F^*)$  is semistable, or equivalently,   $(V, \nabla)$ admits a gr-semistable  Griffiths-transverse filtration.
\end{theorem}

\begin{remark}
 By \cite{Weil}, over a complex curve,  every holomorphic vector bundle  that  admits a connection  is of degree $0$,  thus in this case every flat bundle $(V,\nabla)$ is  naturally $\nabla$-semistable, and \cite[Theorem 2.5]{Simpson} is a special case of  Theorem \ref{goodfil}  when $k=\C$. When $k$ is of characteristic $p>0$, the degree of  a vector bundle  that  admits a connection  is  $0$ after modulo $p$ but not necessarily  $0$, thus not every flat bundle is $\nabla$-semistable.
\end{remark}

To prove the above theorem, we need some lemmas.

\begin{lemma}\label{maxsubbundle}
Let $(E, \theta)$ be a system of Hodge bundle, i.e. $E=\oplus_{i=0}^n E^i$ with $\theta: E^i\to E^{i-1}\otimes_{\sO_X}\Omega^1_X$. If $(E, \theta)$ is unstable as a  Higgs bundle, then its  maximal semistable Higgs subbundle $I\subset E$  is itself a system of Hodge bundles, that is   $I=\oplus_{i=0}^n I^i$ with $I^i:=I\cap E^{i}$.
\end{lemma}
\begin{proof}
 Choose $t\in k$ such that $t^i\neq 1 $ for $0<i\leq n$. Note that there is  an isomorphism $f: (E, \theta)\to (E, t\theta)$ giving by $f|_{E^i}=t^i\text{id}$. Because of  the uniqueness of the  maximal destabilizing subobject, we see that $f(I)=I$.
Let $s$ be any local  section of $I$. Write  $s$ as  $\sum_{i=0}^ns^i$, where $s^i$ is a local section of $E^i$. Then
\begin{align*}
\text{for }j\geq0,
f^{j} (s)=\sum_{i=0}^nt^{ji}s^i \in I.
\end{align*}
Consider
 $$
 \begin{pmatrix}
s\\
f(s)\\
\vdots\\
f^n(s)
 \end{pmatrix} =
 \begin{pmatrix}
   1 & 1 &1 &\cdots &1 \\
   1 & t &t^2&\cdots &t^n\\
   \vdots & \vdots& \vdots& \ddots & \vdots \\
   1 & t^n& t^{2n}& \cdots &t^{n^2}
 \end{pmatrix} \cdot
 \begin{pmatrix}
s^0\\
s^1\\
\vdots\\
s^n
 \end{pmatrix}.
$$
 By assumption on $t$, the coefficient matrix is invertible; thus all $s^i$'s are local sections of $I$ and $I =\oplus_{i=0}^n I\cap E^{i}$.
\end{proof}

Let $(V, \nabla)$ be a flat bundle over $X$. Start with an arbitrary Griffiths-transverse filtration $F^*$ of weight $n$ and consider the associated Higgs bundle $(\text{Gr}_F(V), \theta)$. If  $(\text{Gr}_F(V), \theta)$  is unstable,  let $I_F$   be  the  maximal semistable Higgs subbundle, which is also called the maximal destabilizing subobject; otherwise let $I_F=\text{Gr}_F(V)$. By Lemma \ref{maxsubbundle},
\begin{align*}
I_F=\oplus_{i=0}^nI_F^{i}, \quad  I_F^i\subset F^i/F^{i+1} \subset V/F^{i+1}.
\end{align*}
We define an operation $\xi$ on the set of  Griffiths-transverse filtrations. The new filtration $\xi(F)^*$ of $V$ is given by
\begin{align*}
\xi(F)^{i+1}:=\text{Ker}(V\to  \frac{ V/F^{i+1}}{I_F^i} ), \text{ for }   0\leq i\leq n; \quad \xi(F)^0=V.
\end{align*}
Note that $F^i\supset \xi(F)^{i+1}\supset F^{i+1}$ and  $\xi(F)^*$ is again Griffiths-transverse.  And there is  an exact sequence
\begin{align}
0\to \text{Gr}^i_F(V)/I_F^i\to \text{Gr}^i_{\xi(F)}(V) \to I_F^{i-1} \to 0, \quad \text{for }  0\leq i\leq n+1.
\end{align}
By adding all together, we obtain an exact sequence of systems of Hodge bundles
\begin{align}\label{exactseq}
0\to \text{Gr}_F(V)/I_F\to \text{Gr}_{\xi(F)}(V) \to I_F^{[1]} \to 0,
\end{align}
where $E^{[k]}$ denotes the system of Hodge bundle $E$ with Hodge index shifted so that $(E^{[k]})^i=E^{i-k}$. If $(E, \theta)$ is unstable, let $\mu_\text{max}(E)$ and $r_\text{max}(E)$ denote respectively the slope and rank of the maximal destabilizing subobject of $E$; otherwise, let $\mu_\text{max}(E)=\mu(E)$ and $r_\text{max}(E)=\text{rk}(E)$.  By (\ref{exactseq}), we have

\begin{lemma}\label{invariants}
\begin{enumerate}
\item $\mu_\text{max}(\text{Gr}_{\xi(F)}(V)) \leq  \mu_\text{max}(\text{Gr}_F(V)) $.

\item  If $\mu_\text{max}(\text{Gr}_{\xi(F)}(V)) =  \mu_\text{max}(\text{Gr}_F(V)) $, then $r_\text{max}(\text{Gr}_{\xi(F)}(V)) \leq r_\text{max}(\text{Gr}_F(V)) $.

\item  Moreover, if $r_\text{max}(\text{Gr}_{\xi(F)}(V)) = r_\text{max}(\text{Gr}_F(V)) $, then (\ref{exactseq}) is split as Higgs bundles and $ I_F^{[1]}$ is  the maximal destabilizing subobject of $\text{Gr}_{\xi(F)}(V)$, denoted by $ I_{\xi(F)}= I_F^{[1]}$.
\end{enumerate}
\end{lemma}

We will see that either $\mu_\text{max}$ or $r_\text{max}$ drops after applying $\xi$ finitely many times in the case of a  flat  $\nabla$-semistable vector bundle.
\begin{lemma}\label{terminal}
Let $F^*$ be a weight-$n$ Griffiths-transverse filtration of a flat  $\nabla$-semistable vector bundle $(V,\nabla)$ over $X$. Assume that $(\text{Gr}_F(V), \theta)$ is an unstable Higgs bundle.   Then at least one of the following two inequalities holds:
\begin{align*}
(1) \mu_\text{max}(\text{Gr}_{\xi^n(F)}(V))<  \mu_\text{max}(\text{Gr}_F(V)); \quad (2)  r_\text{max}(\text{Gr}_{\xi^n(F)}(V)) < r_\text{max}(\text{Gr}_F(V)).
\end{align*}
\end{lemma}

\begin{proof}
 Let $\mu_k=\mu_\text{max}(\text{Gr}_{\xi^k(F)}(V))$ and $r_k=r_\text{max}(\text{Gr}_{\xi^k(F)}(V))$ for $k\geq 0$. By Lemma \ref{invariants}, $\{(\mu_k, r_k) | k\in \N\}$ decreases in the lexicographic ordering.  Argue by contradiction. Suppose on the contrary that $\mu_n=\mu_0$ and $r_n=r_0$. Then for $0\leq k<n$,  $\mu_{k+1}=\mu_{k}$ and $r_{k+1}=r_k$. By Lemma \ref{invariants},  $I_{\xi^{k+1}(F)}=I_{\xi^{k}(F)}^{[1]}=I_F^{[k+1]}$. It is not hard to check that for some $0< n_0\leq n$,
 the exact sequence (\ref{exactseq}) for  $\text{Gr}_{\xi^{n_0}(F)}(V))$
\begin{align} \label{seq3}
0\to \text{Gr}_{\xi^{n_0-1}(F)}(V)/I_F^{[n_0-1]}\to \text{Gr}_{\xi^{n_0}(F)}(V) \to I_F^{[n_0]} \to 0,
\end{align}
 as graded vector bundles is of the form
\begin{align} \label{seq4}
0\to \oplus_{i=0}^{n'}E^i \to \oplus_{i=0}^{n'+k}E^i \to \oplus_{i=n'+1}^{n'+k} E^{i} \to 0.
\end{align}
Since $\mu_{n_0}=\mu_{n_0-1}$ and $r_{n_0}=r_{n_0-1}$, (\ref{seq3})
is split as   Higgs bundles; as   (\ref{seq3}) is of the form (\ref{seq4}),  $  I_F^{[n_0]}$ corresponds to a  subbundle  of $V$  preserved by  $\nabla$ with slope $\mu_{n_0} >\mu(V)$, contradiction with the $\nabla$-semistability of $(V,\nabla)$.
\end{proof}

\begin{proof}[Proof of Theorem \ref{goodfil}] Choose an arbitrary  Griffiths-transverse filtration $F^*$ of $(V, \nabla)$.
For $ k\in \N $, let  $\mu_k=\mu_\text{max}(\text{Gr}_{\xi^k(F)}(V))$ and $r_k=r_\text{max}(\text{Gr}_{\xi^k(F)}(V))$.
 By Lemma \ref{invariants}, $\{(\mu_k, r_k)| k\in \N \}$ decreases  in the lexicographic ordering; By Lemma \ref{terminal}, for every $(\mu_k, r_k)$,      if $\mu_k>\mu(V)$, there exists some $N>0$ such that $(\mu_{k+N}, r_{k+N})< (\mu_k, r_k)$. As both $\mu_k$ and  $ r_k$  can take on only finitely many values, there exists some $k_0>0$ such that the subsequence $\{(\mu_k, r_k)| k\geq k_0 \}$ is constant,  which implies that $\mu_{k_0}=\mu(V)$. In other words, $( \text{Gr}_{ \xi^{k_0}(F)}(V), \theta)$ is a semistable Higgs bundle. The Griffiths-transverse filtration $\xi^{k_0}(F)^*$ of $(V, \nabla)$ is as desired.
\end{proof}

 \section{Semistable Higgs bundles are strongly Higgs  semistable}
We give a proof of the conjecture that semistable Higgs bundles are strongly Higgs  semistable in small ranks.
 In this section, $k=\overline{\F}_p$ for some odd prime $p$.

 We first recall some notations from  \cite{LSZ}.With a smooth $W_2$-lifting of $X$, one can define   an  inverse Cartier functor  $C_0^{-1}$  from the category of (semistable) Higgs bundles over $X$ with nilpotent Higgs field of exponent  $<p$ to the category of ($\nabla$-semistable) flat bundles.
A Higgs-de Rham flow is a sequence of the following form:
$$
\xymatrix{
                &  (H_0,\nabla_0)\ar[dr]^{Gr_{F_0}}       &&  (H_1,\nabla_1)\ar[dr]^{Gr_{F_1}}    \\
 (E_0,\theta_0) \ar[ur]^{C_0^{-1}}  & &     (E_1,\theta_1) \ar[ur]^{C_0^{-1}}&&\ldots.       }
$$
If all  $(H_i, \nabla_i)$'s and $(E_i, \theta_i)$'s in a flow are defined over a finte field $k_0\supset \F_p$, then we say that  the flow is defined over   $k_0 $.
 A Higgs bundle is called to be strongly Higgs semistable, if it  appears as the leading term of a Higgs-de Rham flow defined over a finite field.
\begin{lemma}\label{unique fil} Let $(V,\nabla)$ be  a flat vector bundle  of degree zero. If there exists a   Griffiths-transverse filtration $F^*$ such that  the associated-graded Higgs bundle $(E,\theta)$ is stable, then  such a filtration is unique up to index shifting. 
\end{lemma}
\begin{proof}
Suppose that there is another  gr-semistable Hodge filtration  $\bar{F}^*$   different from $F^*$ up to index shifting.\\  
  
Case 1: Suppose that there exist an integer $N$, such that for every $i$, $F^{i+N}\subset \bar{F}^i$. Then the inclusion induces a natural  map  $f:(E,\theta)\to (\bar{E},\bar{\theta})$. As the two filtrations are different, $f$ is neither  injective  nor a zero map. On one side, $Im(f)$ is a quotient object of $(E,\theta)$, $\mu(Im(f))>0$; on the other side, $Im(f)$ is also a subobject of $(\bar{E},\bar{\theta})$, $\mu(Im(f))\leq0$, contradiction.

Case  2: otherwise,
  let $a$ be the largest integer such that $F^a$ is not contained in $\bar{F}^a$. And $b$ be the largest integer such that $F^{a-i}$ is contained in $F^{b-i}$ for  all $i\geq 0$. Then $b<a$. And we can define a  map $f:(E,\theta)\to (\bar{E},{\theta})$ as follows: for $i>0$, $f|_{E^{a+i}}=0$ , and for $j\leq 0$ $f|_{E^{a-j}}$ is
        $$
          E^{a-j}=F^{a-j}/F^{a-j+1}\to \bar{F}^{b-j}/\bar{F}^{b-j+1}.
       $$
   We see that $f$ is neither  injective  nor a zero map, similarly as in Case 1, we obtain a contradiction.

\end{proof}

\begin{theorem}
Let $(E_0,\theta_0)$ be a semistable Higgs bundle defined over a finite field $k_0$ with nilpotent Higgs field of exponent  $<p$. Assume that   $\text{rk}(E_0)<p$, then $(E_0,\theta_0)$ is  strongly Higgs semistable.
\end{theorem}

\begin{proof} Let $(H_0,\nabla_0):=C_0^{-1}(E_0,\theta_0)$. Note that $(H_0,\nabla_0)$ is $\nabla$-semistable and also defined over $k_0$.  By Theorem \ref{goodfil}, there exists a gr-semistable Griffiths-transverse filtration $F^*$ of $(H_0,\nabla_0)$. Because  $\text{rk}(E_0)<p$,    $(E_1, \theta_1):=(\text{Gr}_F(H_0), \theta)$ satisfies the same assumptions as that on $(E_0,\theta_0)$. We can apply the two functors $C_0^{-1}$ and $\text{Gr}$ alternatively and obtain a Higgs-de Rham flow with $(E_0,\theta_0)$ to be the leading term. 
It suffices to  o show that  such a Higgs-de Rham flow is  defined over a finite field. Argue by induction on $\text{rk}(E_0)$. 
Assume that the statement is  true when $\text{rk}(E_0)<r$ for $r>1$.
Now suppose that  $(E_0,\theta_0)$ is of rank $r$.

Case 1: $(E_1, \theta_1)$ is  not stable. The filtration  $F^*$  is defined over a finite extension $k_1$ of $k_0$, so is $(E_1,\theta_1)$. Since $(E_1,\theta_1)$ is not stable, there exists an extension of Higgs bundles
$$0\to (E_1',\theta_1')\to (E_1,\theta_1)\to (E'',\theta'')\to 0 $$
over a finite extension $k_2$ of $k_1$  with the same slope (=0).\
By induction assumption  we can assume that  both $(E_1',\theta_1')$ and $(E_1, \theta_1'')$ are  leading terms of flows defined over a finite field $k_3\supset k_2$.
Applying inverse Cartier functor over  the extension we obtain an extension of flat bundles over $k_3$.
$$0\to (H_1',\nabla')\to (H_1,\nabla_1)_1\to (H_1'',\nabla_1'')\to 0.$$
Let $(F')^*$ (resp. $(F'')^*$) denote the gr-semistable hodge filtraion  defined over  $k_3$  of $H_1'$ (resp. $H_1''$) in the flows, and the associated graded Higgs bundle is $(E_2',\theta_2')$ (resp. $(E_2'',\theta_2'')$).
One can now make extension of  $(F')^*$ and $(F'')^*$ to obtain
a gr-semistable Hodge filtration on $H_1$ as follows:
pulled back $(F'')^*$ via the projection $H_1\to H''_1$ one obtains
a hodge filtration $(\tilde{F}'')^*$ of $H_1$, and each term of $(\tilde {F}'')^*$ contains $H_1'$. So we can add the $(F')^*$ to $(\tilde{F}'')^*$ and  form a new Hodge filtration  $(F_1)^*$ of $H_1$ defined over $k_3$ .
One checks that it is a gr-semistable Hodge Filtraion on $H_1$, and  the associated graded Higgs bundle $(E_2,\theta_2)$ is a direct sum of $(E_2',\theta_2')$ and $(E_2'',\theta_2'')$.
As invere cartier preserve direct sum, we can produce further terms using the datum of the two flows for  $(E_1',\theta_1')$ and $(E_1'', \theta_1'')$, and all  the terms are defined over $k_3$. 

Case 2: $(E_1, \theta_1)$ is  stable. By lemma(\ref{unique fil}) it has only one gr-semistable Hodge filtration, so the filtration  must be defined over $k_0$. One go on the inverse Cartier-grading process, if in some step it run into case 1, then we are done. If not, then all the Higgs terms are stable, and the filtration are all defined over $k_0$, so the flows is already defined over $k_0$.
\end{proof}
$ $\\
As already mentioned in \cite{Li} , we have the following corollary.
\begin{corollary}
 Let k be the algebraic closure of finite fields of characteristic $p >0$, and $X$ a smooth projective varieties over $k$ which has a $W_2(k)$-lifting.  If $(E_i,\theta_i)$ for $i=1,2$ are semistable Higgs bundles of degree $0$, and  $rank(E_1)+rank(E_2)<p$, then the tensor product $(E_1\otimes E_2, \theta_1\otimes1 +1\otimes \theta_2)$ is also semistable.
\end{corollary}

\end{document}